\documentclass[10pt,a4paper]{article}
\usepackage[onehalfspacing]{setspace}
\usepackage[hidelinks]{hyperref}
\usepackage[utf8]{inputenc}
\usepackage[a4paper, right=2cm, textwidth=16cm, bottom=3.3cm]{geometry}
\usepackage[T1]{fontenc}
\usepackage[english]{babel}
\usepackage{amsfonts}
\usepackage{amsmath,mathtools,amsthm,amssymb,amstext}
\usepackage{amsfonts}
\usepackage{dsfont}
\usepackage{xcolor}
\usepackage{comment}
\usepackage{tikz}
\usepackage{caption}
\usepackage{subcaption}
\usepackage{commath}
\usepackage{pgfplots}
\usepackage{scalerel}
\usetikzlibrary{patterns}
\usepackage[outline]{contour}
\usepackage{todonotes}
\usepackage{ifthen}
\usepackage{xkeyval}
\usepackage{calc}
\usepackage{extarrows}
\usepackage{authblk}
\usepackage{listings}

\contourlength{0.4pt}

\theoremstyle{definition}
\newtheorem{theorem}{Theorem}[section]

\newtheorem{lemma}[theorem]{Lemma}

\newtheorem{notation}[theorem]{Notation}

\newtheorem{remark}[theorem]{Remark}
\newtheorem{assumption}[theorem]{Assumption}

\newcommand{\y}[1]{\lVert #1 \rVert}
\newcommand{\Y}[1]{\left\lvert#1\right\rvert}

\newcommand{\under}[2]{\underset{#1}{\underbrace{#2}}}

\newcommand{\R}{\mathbb{R}}
\newcommand{\Z}{\mathbb{Z}}
\newcommand{\E}{\mathbb{E}}

\newcommand{\Raeps}{\mathcal{R}_\varepsilon}

\newcommand{\Reps}{\mathcal{R}^*_\varepsilon}
\newcommand{\Repsu}{\mathcal{R}^*_\varepsilon u_\varepsilon}

\newcommand{\p}{2}
\newcommand{\Eepsu}{\mathcal{E}_\varepsilon u_\varepsilon}

\newcommand{\cweighta}[1]{\varpi \left( #1 \right)}

\newcommand{\ocweight}{\varpi}

\newcommand{\Jepsu}{\mathcal{E}^{V}_\varepsilon u_\varepsilon}

\newcommand{\Eloc}{\mathcal{E}^{loc}_\varepsilon u_\varepsilon }

\newcommand{\asarrow}{\xlongrightarrow[\text{a.s.}]{\varepsilon \rightarrow 0} }
\newcommand{\epsarrow}{\xlongrightarrow{\varepsilon \rightarrow 0} }

\newcommand{\f}{w}

\newcommand{\coeff}{\sigma_{\varepsilon, x, y}}
\newcommand{\B}{B}
\newcommand{\q}{\ell}
\newcommand{\Ve}{V(x,y,u_\varepsilon (x)-u_\varepsilon (y))}
\newcommand{\V}{V(x,y,u(x)-u(y))}
\newcommand{\RV}{V(\Reps x,\Reps y,\Repsu (x)-\Repsu (y))}
\newcommand{\RVk}{V(\Reps x,\Reps y,\Repsu^k (x)-\Repsu^k (y))}
\newcommand{\Vk}{V(x,y,u^k(x)-u^k(y))}
\newcommand{\var}[1]{\operatorname{Var}\left[ #1 \right]}

\newcommand{\cE}{{\mathcal E}}
\newcommand{\cA}{{\mathcal A}}
\newcommand{\eps}{\varepsilon}
\newcommand{\el}{\mathrm{el}}
\newcommand{\Erm}{\mathrm{E}}
\newcommand{\longrm}{\mathrm{long}}
\newcommand{\localrm}{\mathrm{local}}
\newcommand{\fs}{\mathfrak{s}}

\def\Xint#1{\mathchoice
{\XXint\displaystyle\textstyle{#1}}%
{\XXint\textstyle\scriptstyle{#1}}%
{\XXint\scriptstyle\scriptscriptstyle{#1}}%
{\XXint\scriptscriptstyle\scriptscriptstyle{#1}}%
\!\int}
\def\XXint#1#2#3{{\setbox0=\hbox{$#1{#2#3}{\int}$}
\vcenter{\hbox{$#2#3$}}\kern-.5\wd0}}

\def\dashint{\Xint-}
\newcommand{\uproman}[1]{\uppercase\expandafter{\romannumeral#1}}

\definecolor{dblue}{RGB}{47,82,143}
\definecolor{grey}{RGB}{118,113,113}

\author[1,2]{\small Patrick Dondl \thanks{patrick.dondl@mathematik.uni-freiburg.de}}
\author[3]{Martin Heida \thanks{heida@wias-berlin.de}}
\author[1,2]{Simone Hermann \thanks{simone.hermann@mathematik.uni-freiburg.de}}
\title{Non-local homogenization limits of discrete elastic spring network models with random coefficients}
\date{}

\affil[1]{\footnotesize 
{Cluster of Excellence livMatS @ FIT – Freiburg Center for Interactive
Materials and Bioinspired Technologies,
University of Freiburg,
Georges-Köhler-Allee 105, 79110 Freiburg, Germany}}
\affil[2]{Department for Applied Mathematics, University of Freiburg, Hermann-Herder-Str. 10, 79104 Freiburg, Germany}
\affil[3]{Weierstrass Institute for Applied Analysis and Stochastics, Mohrenstr. 39, 10117 Berlin, Germany}

\allowdisplaybreaks

\begin{document}
\maketitle


\maketitle

\begin{abstract}
This work examines a discrete elastic energy system with local %
interactions described by a discrete second-order functional in the symmetric gradient and additional non-local random long-range interactions. We analyze the asymptotic behavior of this model as the grid size tends to zero.
Assuming that the occurrence of long-range interactions is Bernoulli distributed and depends only on the distance between the considered grid points, we derive -- in an appropriate scaling regime -- a fractional $p$-Laplace-type term as the long-range interactions' homogenized limit.
A specific feature of the presented homogenization process is that the random weights of the $p$-Laplace-type term are non-stationary, thus making the use of standard ergodic theorems impossible.
For the entire discrete energy system, we derive a non-local fractional $p$-Laplace-type term and a local second-order functional in the symmetric gradient.
Our model can be used to describe the elastic energy of standard, homogeneous, materials that are reinforced with long-range stiff fibers.
\end{abstract}

\section{Introduction}
In the field of materials science and mechanics, understanding the elastic properties of a considered material and its energy distribution is essential. In the case of homogeneous or heterogeneous materials with purely local interactions there already exists a wide range of models for the elastic energy. For inhomogeneous materials with non-local components, these models are generally no longer valid.
We focus on a specific model where a homogeneous material is interspersed with randomly distributed long-range elastic rods. Our model is inspired by the complex material system comprising the peel of pomelo fruit, which consists of a soft, elastic foam material interspersed with stiff fibers. As opposed to a typical short-fiber reinforced material, the fibers in the pomelo peel are of a length-scale comparable to that of the entire system. Due to the excellent shock absorption properties of this citrus peel, it makes for an interesting subject for bioinspired technological applications \cite{Buehrig-Polaczek_2016, Thielen2, Pomelo-inspired}. To gain a better understanding of the resulting elastic properties of such random long-fiber reinforced materials we seek to obtain a suitable homogenization scaling limit.

In comparison to standard elastic energy models the mathematical investigation of such materials is much more complex.
We use discrete spring network models which provide a useful framework for the description of many complex elastic phenomena.
To capture the local behavior of the homogeneous material as well as the non-local behavior of the long-range interactions, we introduce a discrete elastic energy functional which consists of a quadratic functional in the discrete symmetric gradient and add an appropriately scaled energy modeling the interaction between the displacement at a given point and some, randomly selected,  counterparts further across the material, reflecting the influence of distant interactions on the energy distribution. 
Our model thus captures both local and non-local effects, providing a comprehensive framework for analyzing the elastic energy distribution in heterogeneous materials subject to random long-range interactions. 

Since local continuum limits of discrete systems have been studied extensively by many authors (see for example \cite{Alicic,Braides2,Braides1,Heida1}), the focus of this work is on the non-local part modeling random interactions between distant points in the material. 
A first approach of a discrete random conductance model with long-range interactions is given in \cite{Heida1}. It showed that if the weights on the far reaching interactions is too low, the limit equation localizes. However, putting more weight on the far reaching interaction moments and assuming that the random field is stationary ergodic, the authors of \cite{Heida2} show that the homogenized limit is given by a fractional Laplace-type term.
However, these assumptions force the weights of individual differences to
decay polynomially with distance. 
Below we will derive a discrete model for elasticity that shares many features of the discrete model \cite{Heida2} but with some important differences:
while in \cite{Heida2} the non-local coefficient field was stationary in both variables and positively bounded from below, our new coefficient field is stationary only in the first variable and it is either $0$ or $1$ with the possibility for $1$ decreasing polynomially in the far distance. This change is necessary to model the long-range interactions inspired by the fiber reinforced material of the pomelo peel, as of course only a comparatively small number of points are connected via such fibers.

From a mathematical perspective, the lower bound on the far field coefficient field made it possible to derive and apply discrete Sobolev-Poincar\'e inequalities for discrete fractional Sobolev spaces and thus no local term was needed in \cite{Heida2} to achieve a uniform $L^p$-bound on the sequence of minimizers. In the present paper, the non-local coefficients dominantly being $0$, we need the local elasticity from a mathematical perspective to achieve enough uniform integrability of solutions to pass to the limit. From another point of view, the stationarity in \cite{Heida2} allowed a relatively straight forward passage to the limit in the coefficient field. However, the break down of stationarity in our case makes it necessary to look for alternatives, which we find in application of Tschebyscheffs inequality.

We show that minimizers of the discrete energy converge to minimizers of a limit energy functional which are simultaneously solutions of
$$-\nabla\cdot\nabla^\fs u(x) +\int_Q \frac{\partial_3V(x,y,u(x)-u(y))}{|x-y|^{d+ps}}\,\mathrm{d} y =f(x)\,.$$

The presented model can be used, for example, as a simplified model for a homogeneous material interspersed with randomly distributed fibers whose material properties are constant and, in particular, do not depend on their length.

\section{Mathematical model and the main result}
We define our discrete model on the re-scaled lattice $\mathbb{Z}_\varepsilon^d := \varepsilon \mathbb{Z}^d$, where $0<\varepsilon <1$.
For a bounded domain $Q \subset \mathbb{R}^d$ we define the rescaled grid $Q_\varepsilon := Q \cap \mathbb{Z}_\varepsilon^d$. 
Furthermore, we assume that there is a non-negative conductance $\ocweight(x,y)$ between any two points $x,y\in Q_\varepsilon$, which defines the interactions between $x$ and $y$ through which energy or information can flow.

To introduce difference-type operators on the discrete grid $Q_\varepsilon$, we use discrete functions $u_\varepsilon : Q_\varepsilon \rightarrow \mathbb{R}^d$ and extend them by zero to $\Z_\varepsilon^d$. In the following $u_\varepsilon$ describes the elastic deformation of each point in $Q_\varepsilon$ and $f_\varepsilon$ denotes the force acting on the material.

\subsection{Elastic spring and rod network model} \label{section2.1}

 When a spring or an elastic rod is subjected to external forces, it experiences changes in its length and shape, leading to internal stress described by the Cauchy stress formula
\begin{equation}\label{eq:nonlin-elast}
    \Erm_\longrm(\delta u,b) = \frac{\left|b + |b|\,\delta u\right| - |b|}{|b|},
\end{equation}
where $b$ is the vector representing the initial distance between two points of the spring, and $\delta u$ is the displacement gradient, given by $\delta u = \frac{u(x + b) - u(x)}{\Y{b}}$. This formulation of strain measures the relative extension of the spring, considering both the original and deformed states.

This approach aligns with Hooke's Law in linear elasticity, which states that the stress is directly proportional to the strain within the elastic limit of the material:
\begin{equation}
    \Sigma = c_\Sigma \Erm,
\end{equation}
where $\Sigma$ is the stress, $c_\Sigma$ is the Young's modulus of the material, and $\Erm$ is the strain. 

For small displacements, the Cauchy stress can be linearized, assuming that the displacements are much smaller than the characteristic dimensions of the object. The linearized stress is given by:
\begin{equation}\label{eq:linearized-elast}
    \Erm_\localrm(\delta u,b) \approx \frac{b}{|b|} \cdot \delta u
\end{equation}
This linear approximation simplifies the analysis but is valid only for small displacements. In the case of larger displacements or far distances, the linearization looses accuracy, as it fails to capture the non-linear behavior and the possible geometric changes in the material.

For this reason, in the below model we assume that the interaction between neighboring nodes of the spring network is given by linearized elasticity, while the long-range interaction is discribed by an abstract potential $V$ that may represent both $\Erm_\longrm$ or $\Erm_\localrm$.

\subsubsection{Discrete elastic energy}

Spring network models have been used before, e.g. in \cite{neukamm2018stochastic} for upscaling elasticity, and we refer to \cite{neukamm2018stochastic} for their applications in numerical models and analysis. We follow their approach and  
propose the following model based on \ref{eq:linearized-elast} for the local elastic energy of the spring-rod-network:
\begin{equation}\label{eq:discrete-elastic-energy}
\cE^\eps_\el u_\eps:=\varepsilon^{d} \sum_{x\in \Z^d_\varepsilon} \sum_{b\in \B} \Y{ \frac{u_\varepsilon(x+\varepsilon b) - u_\varepsilon(x)}{\eps|b|} \cdot \frac{b}{\Y{b}} }^\p  ,
\end{equation}
where 
$$\B := \left\{ b\in \mathbb{R}^d ~\big\vert ~ b= \lambda_1 e_1 + \lambda_2 e_2 + \ldots + \lambda_d e_d ~:~ \lambda_1, \ldots, \lambda_d \in \{-1,0,1\} \right\}\,.$$ 
$\cE^\eps_\el$ is a discrete version of the elastic energy $\int_Q\nabla^\fs u:\cA :\nabla^\fs u$.
This can be seen
by replacing
$\frac{u_\varepsilon(x+\varepsilon b) - u_\varepsilon(x)}{\eps|b|}$ by the directional derivative $\partial_b u_\varepsilon=\nabla u_\varepsilon \,\frac{b}{|b|}$:%

For $d=2$ and assuming $u_\eps$ is continuously differentiable, definition \eqref{eq:discrete-elastic-energy} turns into
\begin{align}
\cE^\eps_\el u_\eps &\approx \varepsilon^{d} \sum_{x\in \Z^d_\varepsilon} \sum_{b\in \B} \Y{ \frac{b}{\Y{b}}\,\nabla u_\eps \frac{b}{\Y{b}} }^\p\label{eq:elastic-concept}\\
& = \varepsilon^{d} \sum_{x\in \Z^d_\varepsilon} \frac54(|\partial_1u_{\eps,1}|^2+ |\partial_2u_{\eps,2}|^2)+\frac12(\partial_1u_{\eps,1}+\partial_2u_{\eps,2})(\partial_1u_{\eps,2}+\partial_2u_{\eps,1}) \\
& \qquad+\frac12\partial_1u_{\eps,1}\partial_2u_{\eps,2}+\frac14|\partial_2u_{\eps,1}+\partial_1u_{\eps,2}|^2\nonumber ,
\end{align}
which is a coercive second order functional in $\nabla^\fs u_\eps$. We will indeed recall in a rigorous manner that $\cE^\eps_\el u_\eps$ converges to a classical energy of linear elasticity.

\subsubsection{Long-range interaction and super-elasticity}\label{section2.1.2}

In our model, we assume that the medium is interwoven with multiple long fibers or rods. The abundance of rods of length $l$ decreases polynomially with $l$ meaning that long rods are significantly less common than short rods.

Overall, we assume that all fibers have the same stiffness $c$. Considering equations \eqref{eq:nonlin-elast} and \eqref{eq:linearized-elast}, we obtain the following functional for the elastic energy stored in the fibers:
$$\cE_\eps^{V_2}(u_\eps) = \eps^d\sum_{x,y\in Q_\eps} \frac{V_2(x,y,u_\eps(x)-u_\eps(y))}{|x-y|^2}$$

Here, $V_2(x,y,u(x)-u(y))$ represents either $$\Y{\left(u(x)-u(y) \right) \cdot (x-y)}^2$$ or 
$$\Y{\Y{x-y +|x-y|\left(u(x)-u(y)\right)}-\Y{x-y}}^2.$$ However, we will allow for more general forms $V$ below. The essential insight from our analysis is that, for certain distributions of the abundance of long fibers, this part of the energy converges in the $\Gamma$-sense to a functional of the form:
$$\cE^{V_2}(u) = \int_Q \int_Q \frac{V_2(x,y,u_\eps(x)-u_\eps(y))}{|x-y|^{d+2s}}$$

In the scalar context, functionals of this form are well-known and associated with the concept of super-diffusion. Analogously, we henceforth associate this functional with super-elasticity.

Super-diffusion describes a significantly faster spreading of particles compared to diffusion, in a sense that diffusion is driven by Bronwian motion while super-diffusion is driven by Levi flights. In result, super-diffusion even has its own time scale compared to diffusion. 
Similarly, we expect that our limit functional describes the propagation of displacements on a shorter time scale and with a longer range than classical linear elasticity does. This faster propagation of displacements would lead to a decrease in the local gradient of the displacements. It thus stands to reason that the aforementioned exceptional shock absorption properties of the pomelo peel are -- in part -- resulting from a non-locality of its macroscopic elastic properties introduced by the long fibers.

We highlight at this point that our convergence analysis also holds for scalar quantities, i.e. we could equally study a scalar $u_\eps$ with correspondingly modified $V$ and obtain coupled diffusion and super-diffusion models.

\subsection{The analytical model}

We now formulate a more abstract version of the above model, that allows us to study the limit behavior of a larger family of rod and fiber models. 

The functional we consider is given by
\begin{align}
\nonumber
\mathcal{E}_\varepsilon u_\varepsilon :=& \varepsilon^{2d}
\sum_{x \in Q_\varepsilon} \sum_{y \in Q_\varepsilon} c
\varepsilon^{-d-ps-\alpha +\q} \ocweight{\left(\tfrac{x}{\varepsilon}, {\tfrac{y}{\varepsilon}}\right)} \frac{\V}{\Y{x-y}^\q}\\
&+ \varepsilon^{d-\p} \sum_{x\in \Z_\varepsilon^d} \sum_{b\in \B} \Y{ \frac{u_\varepsilon(x+\varepsilon b) - u_\varepsilon(x)}{\Y{b}} \cdot \frac{b}{\Y{b}} }^\p 
- \varepsilon^d \sum_{x\in Q_\varepsilon} f_\varepsilon(x) u_\varepsilon(x),
\label{first energy constant weights}
\end{align} 
where $V : Q\times Q \times \mathbb{R}^d \rightarrow \mathbb{R} $ is convex and continuous in the third argument and satisfies the upper $p$-growth condition:
\begin{align} \label{eq:p_growth}
\forall x,y\in Q ~\exists c_{xy}\in \mathbb{R}:\qquad   0\leq \V \leq c_{xy}\Y{u_\varepsilon(x)-u_\varepsilon(y)}^p.
\end{align}

In order to get a representation of a weighted discrete fractional-Laplace-type term in the first expression of \eqref{first energy constant weights} we reformulate it as
\begin{align}
\nonumber
\mathcal{E}_\varepsilon u_\varepsilon =& \varepsilon^{2d}
\sum_{x \in Q_\varepsilon} \sum_{y \in Q_\varepsilon} \coeff \frac{\Ve}{\Y{x-y}^{d+ps}} \\
&+  \varepsilon^{d-\p} \sum_{x\in \Z_\varepsilon^d} \sum_{b\in \B} \Y{ \frac{u_\varepsilon(x+\varepsilon b) - u_\varepsilon(x)}{\Y{b}} \cdot \frac{b}{\Y{b}} }^\p %
- \varepsilon^d \sum_{x\in Q_\varepsilon} f_\varepsilon(x) u_\varepsilon(x),
\label{l:energy constant weights}
\end{align}
where 
$\coeff:= c \varepsilon^{-\alpha} \ocweight{\left(\tfrac{x}{\varepsilon},\frac{y}{\varepsilon}\right)} \left(\tfrac{\Y{x-y}}{\varepsilon}\right)^{d+ps-\q} $
and we abbreviate it as
\begin{align*}
\Eepsu = \Jepsu + \Eloc - \mathcal{F}_\varepsilon u_\varepsilon,
\end{align*}
where
\begin{align*}
\Jepsu &= \varepsilon^{2d}
\sum_{x \in Q_\varepsilon} \sum_{y \in Q_\varepsilon} \coeff \frac{\Ve}{\Y{x-y}^{d+ps}} , \\
\Eloc  &= \varepsilon^{d-\p} \sum_{x\in \Z_\varepsilon^d} \sum_{b\in \B} \Y{ \frac{u_\varepsilon(x+\varepsilon b) - u_\varepsilon(x)}{\Y{b}} \cdot \frac{b}{\Y{b}} }^\p , \\
\mathcal{F}_\varepsilon u_\varepsilon &= \varepsilon^d \sum_{x\in Q_\varepsilon} f_\varepsilon(x) u_\varepsilon(x).
\end{align*} 
The corresponding limit functional is given by
\begin{align} \label{eq:limit_func}
\mathcal{E}u :=& \bar{c} %
\int_Q \int_Q \frac{\V}{\Y{x-y}^{d+ps}} 
+  \int_Q \sum_{b\in B}\left( \frac{b}{|b|}\cdot \nabla u\frac{b}{|b|} \right)^\p - \int_{Q} u(x) f(x) \\
\nonumber
=& \mathcal{E}^{V}u + \mathcal{E}^{loc}u - \mathcal{F}u .
\end{align}
At first sight, this energy, in particular the representation in \eqref{l:energy constant weights}, looks like the one in \cite{Heida2}.
However, we are in a completely different setting, since we impose different assumptions on our random variables $\ocweight$ and thus on the random weights  $\coeff$.

\begin{assumption} \label{m:assumption}
We assume, that $\ocweight (x,y) : \Z^d \times \Z^d \rightarrow \{0,1 \}$ are i.i.d. Bernoulli random variables with probabilities
\begin{align} \label{eq:defprob}
p_{x,y}^{\alpha, \q} := \mathbb{P}\left(\ocweight (x , y) =1 \right) &= 
\begin{cases}
 \tilde{C} \varepsilon^\alpha \Y{x-y} ^{-d-ps+\q}& ~\text{if}~ \Y{x-y} >0 \\
0& ~\text{else},
\end{cases}
\end{align}
where $\tilde{C}\in [0,1]$.
Furthermore we assume that $0<c\in \R$, $d \in \mathbb{N}$, $s \in (0,1)$, 
$\q \in \R$, $\alpha \in \R^+_0$,
$p\in[1,\frac{2d}{d-2})$ if $d>2$ and $p\in[1,\infty)$ if $d=2$.
All parameters must be set such that
\begin{align} \label{cond:param}
d>ps-\q+\alpha
\end{align}
and
\begin{align} \label{cond:param2}
\q < d+ps .
\end{align}
\end{assumption}

Under these assumptions it holds for all $x,y \in \Z^d$ with $\Y{x-y}>0 $ that
\begin{align}
\mathbb{E}\left[\ocweight \left(x,y \right) \right] =
 \tilde{C} \varepsilon^{\alpha} \Y{x-y}^{-d-ps+\q} <\infty. 
\end{align}

\begin{remark}
Note that condition \eqref{cond:param2} is necessary to ensure that the quantities $p_{x,y}^{\alpha, \q}$ as defined in \eqref{eq:defprob} are bounded by one, and thus are indeed probabilities.

Condition \eqref{cond:param} is central to obtain the convergence estimate in  \eqref{eq:conv_prob}. The condition is sharp in the sense that convergence can not be expected without this assumption: the discrete energy $\Jepsu $ can be viewed as a discrete, random, Riemann-sum type approximation to the double integral defining the non-local energy $\mathcal{E}^{V}u$. The expected number of summands for given $\varepsilon$ is of order 
\begin{align*}
    \under{\# \sum_{x \in Q_\varepsilon} \sum_{y \in Q_\varepsilon}  }{\varepsilon^{-2d}} ~\under{\mathbb{P}\left(\ocweight (x , y) =1 \right)}{ \varepsilon^\alpha \varepsilon ^{d+ps-\q}} \approx \varepsilon^{-d+\alpha +ps -\q}.
\end{align*}
$\Jepsu $ can only approximate an integral if the number of its summands tends to infinity as $\varepsilon \rightarrow 0$. This is the case only if $-d+\alpha +ps-\q <0$, which is equivalent to condition \eqref{cond:param}.
\end{remark}

\begin{remark}
As usual in such discrete conductance models, the event $$\ocweight\left(x,y\right)=1$$ can be seen as a connection or a conductance between the points $x$ and $y$. Although it is not mathematically necessary, it is reasonable in applications to assume symmetry for the random variables, i.e. $\ocweight\left(x,y\right) = \ocweight\left(y,x\right) $, which means that the point $x$ is connected with the point $y$ if and only if $y$ is connected with $x$. Especially in the case of fiber reinforced materials, this assumption captures the property that two points are connected via a fiber. The case where the symmetry assumption is not made creates a model where one-way information flow can occur. That is, it can happen that information flows from $x$ to $y$, but not vice versa. However, the homogenization limit is independent of this symmetry assumption. 

In our model we weight connected points with the factor
$c \varepsilon^{-d-ps-\alpha +\q} \Y{x-y}^{-\q}$, which can be interpreted as the strength of the corresponding connection.
In the standard case $\alpha = \q =0$, the probability turns into 
\begin{align*}
p_{x,y}^{0,0} = \tilde{C} \left( \tfrac{\Y{x-y}}{\varepsilon} \right)^{-d-ps}
\end{align*}
and the weight of the connections is 
\begin{align*}
c \varepsilon^{-d-ps}.
\end{align*}
With the parameters $\alpha$ and $\q$ we can further modify the model in a way that we can change the strength of connected points and simultaneously the probability that two points are connected.
How $\alpha$ and $\q$ affect the model can be seen in the following example, where we choose $d=3$, $p=2$ and $s = \tfrac{1}{2}$. This choice of parameters is standard when modeling the elastic energy of a material.

\begin{enumerate}
\item[(1)] For $\alpha = \q = 0$ the probability is $p_1 = \tilde{C} \left( \frac{\Y{x-y}}{\varepsilon} \right)^{-4}$ and the weights are $w_1= c \varepsilon^{-4} $.
\item[(2)] For $\alpha = 0$, conditions \eqref{cond:param} and \eqref{cond:param2} yield that $\q \in (-2, 4)$. The probability is
\begin{align*}
p_{2} = \tilde{C} \left( \tfrac{\Y{x-y}}{\varepsilon} \right)^{-4+\q} ~~
\begin{cases}
~\geq p_1 \qquad \text{for } \q\geq 0\\
~\leq p_1 \qquad \text{for } \q\leq 0
\end{cases}
\end{align*}
 and the weights are 
\begin{align*}
w_2= c \varepsilon^{-4} \left( \tfrac{\Y{x-y}}{\varepsilon} \right)^{-\q} ~~
 \begin{cases} 
~ \leq w_1 \qquad \text{for } \q\geq 0\\
~ \geq w_1 \qquad \text{for } \q\leq 1 .
 \end{cases}
\end{align*} 
Therefore, these parameters create a model where, depending on the sign of $\q$, we can increase the probability and decrease the weights or vice versa. It should also be mentioned that if $\q \neq 0$, the weights depend on the distance of $x$ and $y$. For example, in the case $\q > 0$, the probability is greater than $p_1$, and the weights are smaller than $w_1$, but unlike $ w_1$, the weights $w_2$ decrease with increasing distance $\Y{x-y}$.  
\item[(3)]For $\q = 0$, condition \eqref{cond:param} yields that $\alpha \in [0,2)$.
The probability is 
\begin{align*}
p_{3} = \tilde{C} \varepsilon^\alpha \left( \tfrac{\Y{x-y}}{\varepsilon} \right)^{-4} \leq p_1
\end{align*}
and the weights are 
\begin{align*}
w_3= c \varepsilon^{-4-\alpha} \geq w_1 .
\end{align*}
Thus, by increasing $\alpha$, we get fewer connections with higher weights.
\end{enumerate}
\end{remark}

\subsection{Main Theorem}

To be able to talk about convergence of a discrete function $u_\varepsilon :Q_\varepsilon \rightarrow \mathbb{R}^d$ to a function $u: Q \rightarrow \mathbb{R}^d$ in a proper sense, we define the operator $\Reps$ mapping discrete functions $u_\varepsilon:Q_\eps\to\R^d$ to piece-wise constant functions $\Repsu : Q \rightarrow \mathbb{R}^d$ by
\begin{align*}
\Repsu (x) := 
u_\varepsilon (z) \qquad\text{if } z\in \mathbb{Z}^d_\varepsilon \text{ and } x \in z+\left(\tfrac{\varepsilon}{2} , \tfrac{\varepsilon}{2}  \right]^d \cap Q .
\end{align*}
This means that the operator $\Repsu $ assigns each point in the $\varepsilon$-cube $z+\left(\frac{\varepsilon}{2} , \frac{\varepsilon}{2}  \right]^d$ the value of $u_\varepsilon$ at the centerpoint $z\in \mathbb{Z}^d_\varepsilon$ (see Figure \ref{fig:Reps1} (bottom)).
The operator $\Reps$ is the adjoint operator of the discretization operator $\mathcal{R}_\varepsilon$ mapping functions $u:Q \rightarrow \R^d$ to discrete functions  $\mathcal{R}_\varepsilon u : Q_\varepsilon \rightarrow \R^d$ by 
\begin{align*}
    \mathcal{R}_\varepsilon u(x) := \dashint_{x+\left(\tfrac{\varepsilon}{2} , \tfrac{\varepsilon}{2}  \right]^d \cap Q } u(y) ~\dif y .
\end{align*}
\begin{remark}\label{remark_1}
From now on we will sometimes rewrite sums over discrete functions into an integral form as follows
\begin{align}\label{eq:integral_identity}
\varepsilon^d \sum_{x \in Q_\varepsilon} u_\varepsilon (x) = \int_Q \Repsu (x)  ~\dif x.
\end{align}
Of course this identity is only valid if $Q = \bigcup_{x \in Q_\varepsilon} [x-\frac{\varepsilon}{2}, x+\frac{\varepsilon}{2}]^d$. For general domains $Q$, this does not hold true (see Figure \ref{fig:Reps1} (top right)). However, since $\lim_{\varepsilon \rightarrow 0} \varepsilon^d \Y{ Q_\varepsilon} = \Y{Q}$, where $\Y{U_\varepsilon} = \sharp \{ U \cap \Z^d_\varepsilon \} $, the error made by using \eqref{eq:integral_identity} vanishes in the limit. 
Therefore, without specifying the error, we will use equation \eqref{eq:integral_identity} for simplicity.
\end{remark}

\begin{figure}
\centering
\includegraphics[width=12cm]{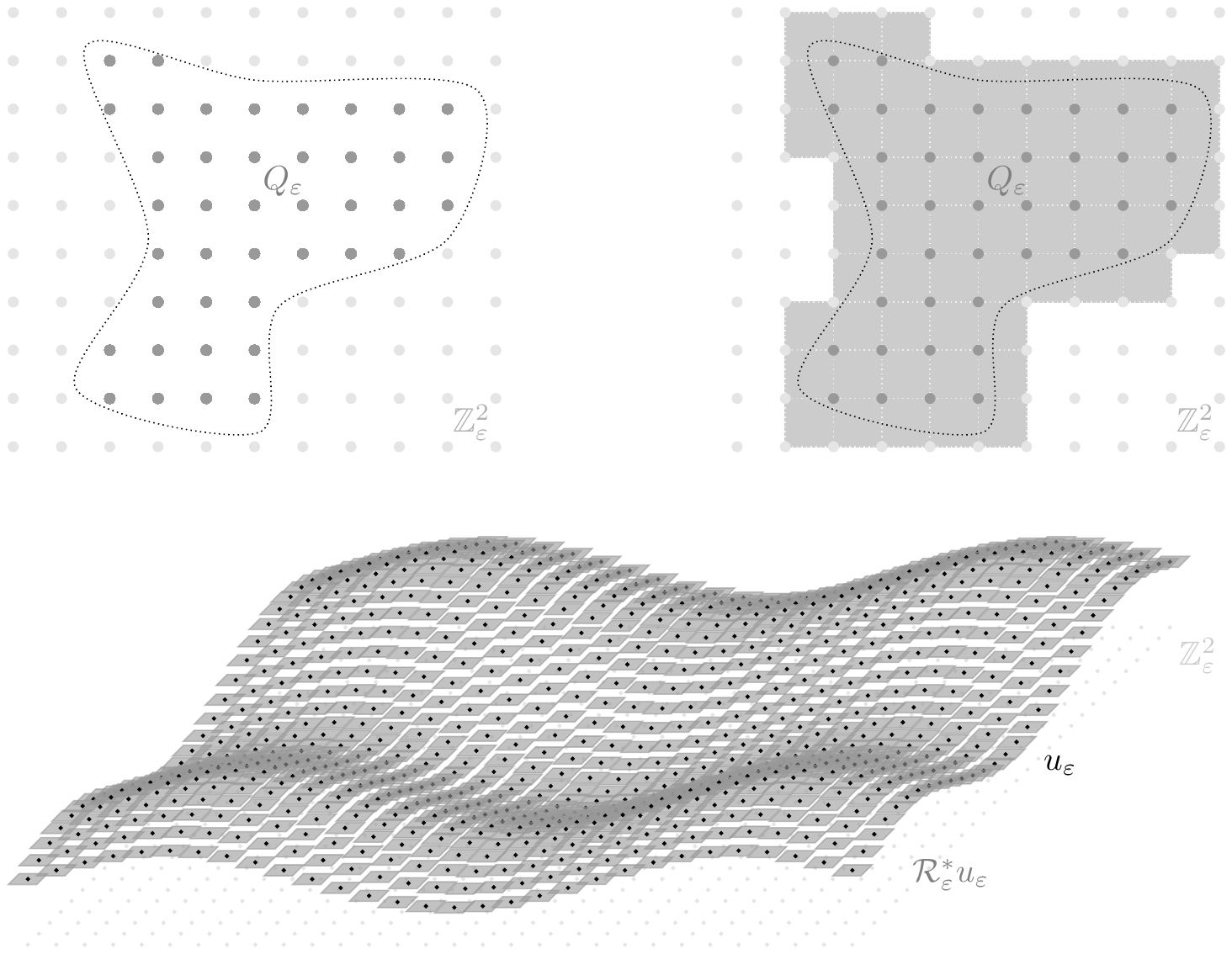}
\caption{TOP: Example of a domain $Q \subset \R^2$, its discretization $Q_\varepsilon = Q \cap \varepsilon \Z^2$ (left) and the difference between $\varepsilon^2 \Y{Q_\varepsilon}$ and $Q$ (right).\\
BOTTOM: The discrete function $u_\varepsilon (x,y)= \frac{1}{10}(\sin(\pi x)+\sin(\pi y))+0.2$ for $\varepsilon = 0.1$ on $Q_\varepsilon =[0,3]^2 \cap \Z_\varepsilon^2$ and its piece-wise constant extension $\Repsu$ to $Q =[0,3]^2$.}
    \label{fig:Reps1}
   \end{figure}

\begin{theorem}\label{th:maintheorem}
Let $Q\subset \R^d$ be a bounded domain with Lipschitz boundary, let $\ocweight , c, d, s, p, \q$ satisfy Assumption \ref{m:assumption}, let $f_\varepsilon : Q_\varepsilon \rightarrow \R^d$ be such that $\Reps f_\varepsilon \rightharpoonup f$ in $L^{2}(Q)^d$. 
Then $\mathcal{E}_\varepsilon$ Mosco-converges to $\mathcal{E}$ in $L^2(Q)^d$ in the following sense:
\begin{enumerate}
\item[(i)] For every sequence $u_\varepsilon : \Z^d_\varepsilon \rightarrow \mathbb{R}^d$ 
with $u_\varepsilon = 0$ on $\Z_\varepsilon^d\setminus Q_\varepsilon$ such that $\Repsu \rightharpoonup u$ weakly in $L^2(Q)^d$,
we have
\begin{align*}
\mathcal{E}u \leq \liminf_{\varepsilon \rightarrow 0 } \mathcal{E}_\varepsilon u_\varepsilon  \qquad \text{a.s.}
\end{align*}
\item[(ii)] For every $u\in L^2(Q)^d$ with $u= 0$ on $\R^d \setminus Q$, there is a sequence $u_\varepsilon$ with $u_\varepsilon = 0$ on $\Z_\varepsilon^d\setminus Q_\varepsilon$ such that $\Repsu \rightarrow u$ strongly in $L^2(Q)^d$ %
and 
\begin{align*}
\mathcal{E}u \geq \limsup_{\varepsilon \rightarrow 0} \mathcal{E}_\varepsilon u_\varepsilon \qquad \text{a.s.}
\end{align*}
\end{enumerate}
\end{theorem}

The difficulty in the proof of Theorem \ref{th:maintheorem} is to show that one can pass to the limit in the non-local term $\Jepsu$.
In contrast to \cite{Heida2} we cannot use ergodic theorems: 
The field $(x,y)\to\varpi(x,y)$ is not jointly stationary in $(x,y)$ but only in $x$. 
This can be seen from the spatial distribution: while $\varpi(x,y)$ has a high density of the value $1$ close to the axis $x=y$, the density of the values $1$ decreases polynomially with increasing $|x-y|$. 
This on the other hand contradicts stationarity, a property that basically states that the distribution of the value $1$ is invariant with respect to shifts in the $(x,y)$-plane. 

In order to compensate for this lack of stationarity, we will resort to Tschebyscheff's theorem instead. This result is well known in probability theory and widely used in the large-deviation community. We will use it to average over all $\varpi(x,y)$, 
observing that this average converges 
polynomially fast to the expectation of $\varpi(x,x-y)$.

\section{Preliminaries}

\subsection{Inequalities} \label{sec3.1}

We will need the following inequalities
\begin{theorem}[Discrete Korn's inequality]\label{thm:Korn}
    There exists a constant $C<\infty$ such that for every  $u_\eps:\, \Z^d_\eps\to \mathbb{R}^d$ with $u_\varepsilon = 0$ on $\Z_\varepsilon^d\setminus Q_\varepsilon$ 
    it holds
    \begin{equation}\label{eq:Korn}
    \eps^d\sum_{x\in Q_\eps}\sum_{i=1}^d \eps^{-2}\Y{u_\eps(x+\eps e_i)-u_\eps(x)}^2\leq C \Eloc .        
    \end{equation}
\end{theorem}
\begin{proof}
    Even though the inequality is well known, let us repeat the main argument following the concept of \eqref{eq:elastic-concept}: As $(u(x+\eps e_i)-u(x))\cdot e_i = u_i(x+\eps e_i)-u_i(x)$, an estimate 
    $$\eps^d\sum_{x\in Q_\eps}\sum_{i=1}^d \eps^{-2}\Y{u_{\eps,i}(x+\eps e_i)-u_{\eps,i}(x)}^2\leq C \Eloc
    $$
    is evident. Next, we recall that $2(x+y)^2+2y^2\geq x^2$ and hence we infer from 
    \begin{align*}
        (u(x+&\eps e_i+\eps e_j)-u(x))\cdot (e_i+e_j) \\
        &= (u_i(x+\eps e_j)-u_i(x)) +(u_i(x+\eps e_i +\eps e_j)-u_i(x+\eps e_j))\\
        &\quad +(u_j(x+\eps e_i)-u_j(x)) +(u_j(x+\eps e_j +\eps e_i)-u_j(x+\eps e_i))
    \end{align*}
    with the choices
    \begin{align*}
        x &= (u_i(x+\eps e_j)-u_i(x))+ (u_j(x+\eps e_i)-u_j(x)) \\
        y &= (u_i(x+\eps e_i +\eps e_j)-u_i(x+\eps e_j)) + (u_j(x+\eps e_j +\eps e_i)-u_j(x+\eps e_i))\\
        x+y &= (u(x+\eps e_i+\eps e_j)-u(x))\cdot (e_i+e_j)
    \end{align*}
    a bound on 
    $$\eps^d\sum_{x\in Q_\eps}\sum_{i,j=1}^d \eps^{-2}A(x):A^\top(x)=\eps^d\sum_{x\in Q_\eps}\sum_{i,j=1}^d \eps^{-2}\Y{a_{ij}(x)}^2\leq C \Eloc\, ,
    $$
    where $B(x)=(b_{ij}(x))_{ij}=(u_i(x+\eps e_j)-u_i(x))$ we define $A=(a_{ij})_{ij}=\frac12(B+B^\top)$. 

    We drop the $\eps$ and make use of $v_j^i(x):=u_j(x+e_i)-u_j(x)$ and write
    \begin{align*}
        \sum_{x\in \Z^d_\eps} &A(x):A^\top(x) = \sum_{x\in \Z^d_\eps} A(x):B(x) \\
        &= \frac12\sum_{x\in \Z^d_\eps}\sum_{i,j=1}^d\left((u_i(x+e_j)-u_i(x))^2 + v_j^i(x)(u_i(x+e_j)-u_i(x))\right)\\
        &= \frac12\sum_{x\in \Z^d_\eps}\sum_{i,j=1}^d(v_i^j)^2 + \frac12\sum_{x\in \Z^d_\eps}\sum_{i=1}^d u_i(x) \sum_{j=1}^d v_j^i(x-e_j)-v_j^i(x) \\
        &= \frac12\sum_{x\in \Z^d_\eps}\sum_{i,j=1}^d(v_i^j)^2 + \frac12\sum_{x\in \Z^d_\eps}\sum_{i=1}^d (u_i(x-e_i)-u_i(x)) \sum_{j=1}^d (u_j(x-e_j)-u_j(x)) \\
        &\geq \frac12\sum_{x\in \Z^d_\eps}\sum_{i,j=1}^d(u_i(x+e_j)-u_i(x))^2\,.
    \end{align*}
    Combining the above estimates results in \eqref{eq:Korn}.
\end{proof}

\begin{theorem}[Discrete Poincar\'e's inequality and compactness] \label{thm:Poincare}
    There exists a constant $C<\infty$ such that for every $\eps>0$, every $p\in[1,\frac{2d}{d-2})$ if $d>2$ or $p\in[1,\infty)$ if $d=2$ and for every  $u_\eps:\, \Z^d_\eps\to \mathbb{R}^d$ with $u_\eps(x)=0$ for $x\not\in Q$ it holds
    \begin{equation}\label{eq:Korn2}
    \left(\eps^d\sum_{x\in Q_\eps}|u_\eps(x)|^p\right)^\frac{1}{p}\leq
    C \left(\eps^d\sum_{x\in \Z^d_\eps}\sum_{i=1}^d \eps^{-2}\Y{u_\eps(x+\eps e_i)-u_\eps(x)}^2 \right)^\frac{1}{2} .     
    \end{equation}
    Furthermore, for every sequence $u_\eps$ where $\eps\to0$ and where 
    $$\sup_\eps\left(\eps^d\sum_{x\in \Z^d_\eps}\sum_{i=1}^d \eps^{-2}\Y{u_\eps(x+\eps e_i)-u_\eps(x)}^2 \right)^\frac{1}{2} <\infty$$ there exists a subsequence $u_{\eps'}$ and $u\in L^p(Q)^d$ such that $\mathcal{R}_{\eps'}^\ast u_{\eps'}\to u$ strongly in $L^p(Q)^d$ as $\eps'\to0$.
\end{theorem}

\begin{proof}
    For the Poincar\'e-inequality refer to \cite{bessemoulin2015discrete} Theorem 6 Section 4.2, while for the compactness refer to \cite{droniou2018gradient}, Lemma B.19.
\end{proof}

The following well known Theorem by P. L. Tschebyscheff is taylored to our coefficient field and replaces the ergodicity assumption in \cite{Heida2}.

\begin{theorem}[Tschebyscheff]\label{th:tschebyscheff}
Let $X$ be a random variable with bounded variance. Then for any $\delta >0$ the following estimates holds true
\begin{align}\label{tschebyscheff0}
\mathbb{P} \left( \Y{X- \E [X]} \geq a \right) \leq \frac{\var{X}}{a^2},
\end{align}
\begin{align}\label{tschebyscheff}
\mathbb{P} \left( \Y{X- \E [X]} < a \right) \geq  1- \frac{\var{X}}{a^2}.
\end{align}
\end{theorem} 

\begin{lemma}[An auxiliary $\delta$-inequality]\label{lem:delta-lemma}
    Let $0<A,B<\infty$. There exist $0<c_1,c_2<\infty$ depending only on $A$ and $B$ such that for every $\delta\in(0,AB)$ and every $a\in(0,A)$, $b\in(0,B)$ with $AB-ab<\delta$ it holds
    $$A-a<c_1\delta\,,\qquad B-b<c_2\delta\,.$$
\end{lemma}
\begin{proof}
    Without loss of generality let $A=B=1$ and $a^\complement=1-a$, $b^\complement=1-b$. It holds:
    \begin{align*}
        \delta>AB-ab=1-ab & =a^\complement b^\complement+a^\complement(1-b^\complement)+(1-a^\complement)b^\complement\\
        & = a^\complement+b^\complement-a^\complement b^\complement
    \end{align*}
    From here we conclude that $a^\complement,b^\complement<\delta$.
\end{proof}

\subsection{Convergence results}

\begin{theorem} 
\label{th:gamma_konv_laplace}
Let $Q\subset \mathbb{R}^d$ be a bounded domain. Then $\cE^{loc}_\eps$ $\Gamma$-converges to $\cE^{loc}$ in the following sense:
\begin{itemize}
    \item[(i)] For every sequence $\eps\to0$ and every sequence $u_\eps:\Z^d_\eps\to\R^d$
    with $u_\varepsilon = 0$ on $\Z_\varepsilon^d\setminus Q_\varepsilon$ such that $\sup_\eps\cE^{loc}_\eps(u_\eps)<\infty$ there exists a subsequence $\eps'$ and $u\in L^2(Q)^d$ such that $\mathcal{R}_{\eps'}^\ast u_{\eps'}\to u$ strongly in $L^2(Q)^d$ and 
    $$\liminf_{\eps'\to0}\cE^{loc}_{\eps'}(u_{\eps'})\geq\cE^{loc}(u)\,.$$
    \item[(ii)] For every $u\in H^1_0(Q)^d$ there exists a sequence $u_\eps$ with $u_\varepsilon = 0$ on $\Z_\varepsilon^d\setminus Q_\varepsilon$ such that $\mathcal{R}_{\eps}^\ast u_{\eps}\to u$ strongly in $L^2(Q)^d$ and 
    $$\limsup_{\eps\to0}\cE^{loc}_{\eps}(u_{\eps})\leq\cE^{loc}(u)\,.$$ 
\end{itemize}
\end{theorem}
\begin{proof}
    Since our given deterministic setting is a particular case of the ergodic setting in \cite{neukamm2018stochastic} we can apply Theorem 4.4 therein to infer for the sequence $\cE^{loc}_\eps$ the existence of some quadratic functional $\cE^{loc}$ such that the claim holds. It remains to determine the exact structure of $\cE^{loc}$. However, this can be achieved by choosing $u\in C^2_c(Q)^d$ and extending it by zero. We then set $u_\eps(x):=u(x)$ for every $x\in Q_\eps$ and make use of $\frac{b}{|b|}\cdot\frac{u(x+\eps b)-u(x)}{\eps|b|}=\eps\frac{b}{|b|}\cdot(\nabla u(x)\frac{b}{|b|}) +O(\eps^2)$ to infer in the limit $\eps\to0$ that $\cE^{loc}$ indeed has the form we provide above. 
\end{proof}

\begin{lemma}
    For a bounded domain $Q\subset\R^d$ there exists $C>0$ such that for every $u\in H^1_0(Q)^d$ and every $\eps>0$ it holds
    \begin{equation}\label{eq:limit-Reps-behavior}
        \left\|\Reps\Raeps u-u\right\|_{L^2(\R^d)^d}\leq C\eps\norm{\nabla u}_{L^2(Q)^d}\,.
    \end{equation}
\end{lemma}
\begin{proof}
    This follows from rescaling the following Poincar\'e inequality to the cube $[0,\eps]^d$
    \begin{equation}
        \left\| u-\dashint_{[0,1]^d}u\right\|_{L^2([0,1]^d)^d}\leq C\norm{\nabla u}_{L^2([0,1]^d)^d}\,.
    \end{equation}
\end{proof}

An important tool for the proof of Theorem \ref{m:l:konvJeps} will be the following Lemma.

\begin{lemma}\label{m:l:konvMittelInt}
Let $Q \subset \R$ be a bounded domain, $\f_\varepsilon : Q \times Q \rightarrow \R^+$ non-negative and suppose that for any open sets $U, J \subset Q$ we have 
\begin{align} \label{m:eq:konvMittelInt}
\dashint_U \dashint_J \f_\varepsilon \epsarrow \tilde{c} ,
\end{align}
where $\tilde{c} \in \R ^+$.
Furthermore let $v_\varepsilon : Q \times Q \rightarrow \R$ be such that  $\sup_{Q\times Q} \Y{v_\varepsilon} \leq C_1< \infty$ and $v_\varepsilon \rightarrow v$ pointwise a.e.
Then 
\begin{align*}
    \int_V \int_W \f_\varepsilon v_\varepsilon \epsarrow \tilde{c} \int_V \int_W v  
\end{align*} 
for any open or compact sets $V,W \subset Q$.
\end{lemma}
\begin{proof}%
First we observe that for a compact set $K\subset Q$ it holds
\begin{align*}
\int_K \int_Q \f_\varepsilon = \int_Q \int_Q \f_\varepsilon - \int_{Q \setminus K} \int_Q  \f_\varepsilon \epsarrow \tilde{c} \Y{Q} \left( \Y{Q} - \Y{Q \setminus K}\right) 
= \tilde{c} \Y{Q} \Y{K}
\end{align*}
and analogously we can argue that for any compact sets $K_1, K_2 \subset Q$ we have the convergence
\begin{align*}
\int_{K_1} \int_{K_2} \f_\varepsilon \epsarrow \tilde{c} \Y{K_1} \Y{K_2}.
\end{align*}
The idea of the main proof is that we split the integral into
\begin{align*}
\int_V \int_W \f_\varepsilon v_\varepsilon = \int_V \int_W \f_\varepsilon (v_\varepsilon -v) + \int_V \int_W \f_\varepsilon v =: L_1^\eps + L_2^\eps
\end{align*}
and show that $L_1^\eps \epsarrow 0$ and $L_2^\eps \epsarrow \tilde{c} \int_V \int_W v$.\\
First we treat $L_1$: 
For any $\delta >0$ there exist, due to Egorov's theorem, compact sets $K_\delta \subset V$, $A_\delta \subset W$ with $\Y{(V\times W)\setminus (K_\delta \times A_\delta)} < \delta$ and such that $v_\varepsilon \rightarrow v$ uniformly on $K_\delta \times A_\delta$.
Moreover due to Lemma \ref{lem:delta-lemma} there exist constants $b_1, b_2 \in \R$ such that $\Y{V \setminus K_\delta} < b_1 \delta$ and $\Y{W \setminus A_\delta} < b_2 \delta$.
We split $L_1^\eps$ into
\begin{align*}
|L_1^\eps| &\leq \int_{V} \int_{W} \f_\varepsilon |v_\varepsilon - v| \\
&\leq \int_{V} \int_{W\setminus A_\delta} \f_\varepsilon |v_\varepsilon - v|
+ \int_{V\setminus K_\delta} \int_{A_\delta} \f_\varepsilon |v_\varepsilon - v|
+ \int_{K_\delta} \int_{A_\delta} \f_\varepsilon |v_\varepsilon - v|.
\end{align*}
For the first term we get from \eqref{m:eq:konvMittelInt} the convergence as follows
\begin{align*}
\int_{V} \int_{W\setminus A_\delta} &\f_\varepsilon |v_\varepsilon - v|
\leq \sup_{V, W\setminus A_\delta} \Y{v_\varepsilon - v} \int_{V} \int_{W\setminus A_\delta} \f_\varepsilon 
\leq  2\sup_{Q} \Y{v_\varepsilon} \int_{V} \int_{W\setminus A_\delta} \f_\varepsilon \\
&\leq 2 C_1 \int_{V} \int_{W\setminus A_\delta} \f_\varepsilon
\epsarrow 2 C_1 \tilde{c} \Y{V} \Y{W\setminus A_\delta} < 2C_1 \tilde{c} \Y{V} b_1 \delta \xlongrightarrow{\delta \rightarrow 0}  0.
\end{align*}
Analogously we get for the second term
\begin{align*}
\int_{V\setminus K_\delta} \int_{A_\delta} \f_\varepsilon |v_\varepsilon - v| 
\leq  2C_1\int_{V \setminus K_\delta} \int_{ A_\delta} \f_\varepsilon 
\epsarrow& 2 C_1 \tilde{c} \Y{V \setminus K_\delta} \Y{A_\delta} \\
&< 2 C_1 \tilde{c} b_2 \delta \Y{W} 
\xlongrightarrow{\delta \rightarrow 0}  0.
\end{align*}
The third term vanishes due to the uniform convergence of $v_\varepsilon$ on $K_\delta \times A_\delta$:
\begin{align*}
\int_{K_\delta} \int_{A_\delta} \f_\varepsilon |v_\varepsilon - v| \leq \sup_{K_\delta, A_\delta} \Y{v_\varepsilon - v} \int_{K_\delta} \int_{A_\delta} \f_\varepsilon \epsarrow 0.
\end{align*}
Since $\delta$ was arbitrary, we conclude $L_1^\eps\to0$ as $\eps\to0$.

Now we treat the remaining term $L_2^\eps$: 
Using Lusin's Theorem, there exist for any $\delta >0$ compact sets $K_\delta \subset V$, $A_\delta \subset W$ such that $v$ is continuous on $K_\delta \times A_\delta$ and $\Y{(V \times W) \setminus (K_\delta \times A_\delta) } < \delta $.
Using again Lemma \ref{lem:delta-lemma}, there exist constants $b_1, b_2 \in \R$ such that $\Y{V \setminus K_\delta} < b_1 \delta$ and $\Y{W \setminus A_\delta} < b_2 \delta$.
Let $v_m = \sum_i v_m^i \chi_{K_{\delta,m}^i \times A_{\delta,m}^i} $ be a sequence of step functions such that $v_m \rightarrow v$ uniformly on $K_\delta \times A_\delta$ and $\y{v_m -v}_\infty < \tfrac{1}{m}$.
It holds 
\begin{align*}
\int_{K_\delta} \int_{A_\delta} \f_\varepsilon v_m = \sum_i v_m^i \int_{K^i_{\delta,m}} \int_{A^i_{\delta,m}} \f_\varepsilon \epsarrow \tilde{c} \sum_i v_m^i \Y{K^i_{\delta,m}} \Y{A^i_{\delta,m}}  = \tilde{c} \int_{K_\delta} \int_{A_\delta} v_m 
\end{align*}

and
\begin{align*}
\Y{\int_{K_\delta} \int_{A_\delta} \f_\varepsilon (v_m-v) } \leq \tfrac{1}{m} \int_{K_\delta} \int_{A_\delta} \f_\varepsilon \epsarrow \tfrac{1}{m} \tilde{c} \Y{K_\delta} \Y{A_\delta}.
\end{align*}
Since 
$$L_2^\eps = \int_V \int_{W\setminus A_\delta} \f_\varepsilon v + \int_{V\setminus K_\delta} \int_{A_\delta} \f_\varepsilon v + \int_{K_\delta} \int_{A_\delta} \f_\varepsilon v_m
+\int_{K_\delta} \int_{A_\delta} \f_\varepsilon (v-v_m) 
$$
the above can be combined to
\begin{align*}
    & \Y{\lim_{\eps\to0}L_2^\eps-\int_{K_\delta}\int_{A_\delta}\tilde c v}\\
    &\quad\leq \sup_{Q\times Q} \Y{v} \tilde{c} \left( \Y{V} \Y{W\setminus A_\delta} + \Y{V\setminus K_\delta} \Y{A_\delta} \right) + \tilde{c} \int_{K_\delta} \int_{A_\delta} |v_m-v| + \tfrac{1}{m} \tilde{c} \Y{K_\delta} \Y{A_\delta}.
\end{align*}
After the limit $m\to\infty$ we infer for any $\delta>0$ small enough
$$\Y{\lim_{\eps\to0}L_2^\eps-\int_{K_\delta}\int_{A_\delta}\tilde c v}\leq C\delta\,,$$
with $C$ independent from $\delta$. From here we conclude $\lim_{\eps\to0}L_2^\eps=\int_{V}\int_{W}\tilde c v$.

\end{proof}

\section{Convergence properties of the non-local energy} \label{chap:first_model}
While Theorem \ref{th:gamma_konv_laplace} will be used to prove the liminf-property as well as the existence of recovery sequences with respect to $\cE^{loc}$, the next theorem will enable us to obtain the same properties for the non-local energy $\Jepsu$ in a proper sense:

\begin{theorem}\label{m:l:konvJeps}
Let $Q\subset \R^d$ be a bounded domain, let $\ocweight , c, d, s, p, \q$ satisfy Assumption \ref{m:assumption} and let $u_\varepsilon : Q_\varepsilon \rightarrow \R^d$ be such that $\Repsu \rightarrow u$ strongly in $L^p$, %
then
\begin{align*}
\liminf_{\varepsilon \rightarrow 0} \Jepsu  &= \liminf_{\varepsilon \rightarrow 0} \int_Q \int_Q \Reps \coeff \frac{\RV}{\Y{\Reps x- \Reps y}^{d+ps}} ~\dif x~\dif y \\
&\geq  \bar{c} %
\int_Q \int_Q \frac{\V}{\Y{x-y}^{d+ps}}  .
\end{align*}
If additionally  $ \sup_\varepsilon \sup_{x,y \in Q_\varepsilon} \frac{\V}{\Y{x-y}^{d+ps}} < \infty $,
then
\begin{align*}
\Jepsu =& \int_Q \int_Q \Reps \coeff \frac{\RV}{\Y{\Reps x- \Reps y}^{d+ps}} ~\dif x~\dif y \\
&\asarrow  \bar{c} 
\int_Q \int_Q \frac{\V}{\Y{x-y}^{d+ps}}  ,
\end{align*}
where $ \bar{c} = c\tilde{C}$ , with $\tilde{C}$ defined in \eqref{eq:defprob} and
\begin{align*}
\Reps \coeff :=  c \varepsilon^\alpha \ocweight{\left( \tfrac{\Reps x}{\varepsilon}, \tfrac{\Reps x}{\varepsilon}- \tfrac{\Reps y}{\varepsilon}\right)} \left(\tfrac{\Y{\Reps x- \Reps y}}{\varepsilon}\right)^{d+ps-\q} .
\end{align*}
\end{theorem}
While the second factor in the integral involving $V$ will converge strongly in $L^1$, the major work to be done relates to the first term. We will prove weak convergence of $\Reps \coeff$ to a constant in order to apply Lemma \ref{m:l:konvMittelInt} in the proof of Theorem \ref{m:l:konvJeps}.
Hence, we will postpone the proof of Theorem \ref{m:l:konvJeps} and first state and proof the weak convergence of $\Reps \coeff$ in Theorem \ref{m:th:konvMittelw}.

Since we want to show convergence of some kind of average integrals, we first give the definition of average integrals of functions that are defined on  a subset of $\mathbb{Z}_\varepsilon^d$:

\begin{notation}
For any $\f_\varepsilon : A_\varepsilon \subset \Z^d_\varepsilon \rightarrow \R$ and any $U_\varepsilon= (U \cap \Z^d_\varepsilon) \subset A_\varepsilon $ we use the notation of average integrals as follows
\begin{align*}
\dashint_U \Reps \f_\varepsilon ( x) ~\dif x := \frac{1}{\Y{U}} \int_U  \Reps \f_\varepsilon ( x) ~\dif x = \frac{1}{\Y{U}} \varepsilon^d \sum_{x \in U_\varepsilon} \f_\varepsilon( x)
\end{align*} 
\end{notation}
Here the second equality results from Remark \ref{remark_1}. 

\begin{theorem}\label{m:th:konvMittelw}
Let $Q\subset \R^d$ be a bounded domain and let $\ocweight , c, d, s, p, \q$ satisfy Assumption \ref{m:assumption}.
Then we have for any open sets $U,J \subset Q$ 
\begin{align*}
\dashint_{x\in U} &\dashint_{y\in J} \Reps \coeff ~\dif y~\dif x \\
=&\dashint_{x\in U} \dashint_{y\in J}  c \varepsilon^\alpha \ocweight{\left( \tfrac{\Reps x}{\varepsilon}, \tfrac{\Reps x}{\varepsilon}- \tfrac{\Reps y}{\varepsilon}\right)} \left(\tfrac{\Y{\Reps x- \Reps y}}{\varepsilon}\right)^{d+ps -\q}  ~\dif y~\dif x 
\asarrow c\tilde{C},
\end{align*}
where $\tilde{C}$ is defined in \eqref{eq:defprob}.
\end{theorem}

\begin{proof}%
Using \eqref{eq:integral_identity}, the integral can be rewritten in the following way
\begin{align*}
 \dashint_{x\in U} &\dashint_{y\in J}  c \varepsilon^\alpha \ocweight{\left( \tfrac{\Reps x}{\varepsilon}, \tfrac{\Reps x}{\varepsilon}- \tfrac{\Reps y}{\varepsilon}\right)} \left(\tfrac{\Y{\Reps x- \Reps y}}{\varepsilon}\right)^{d+ps-\q}  ~\dif y~\dif x\\
 =& c \varepsilon^{-\alpha} \frac{1}{\Y{ U} \Y{J} } \varepsilon^{2d} \underset{x \in U }{\sum}\underset{y \in J_\varepsilon}{\sum} \ocweight{\left(\tfrac{x}{\varepsilon}, \tfrac{x}{\varepsilon}- \tfrac{y}{\varepsilon}\right)}\left( \tfrac{\Y{x-y}}{\varepsilon} \right)^{d+ps-\q} \\
=& c \varepsilon^{-\alpha} \frac{1}{\Y{ U} \Y{J} }  \varepsilon^{2d} \underset{x \in U_\varepsilon }{\sum}\underset{y \in J_\varepsilon}{\sum} 
 \E \left[\cweighta{\tfrac{x}{\varepsilon}, \tfrac{x-y}{\varepsilon}} \right] \left( \tfrac{\Y{x-y}}{\varepsilon} \right)^{d+ps-\q}  \\
&+  c \varepsilon^{-\alpha} \frac{1}{\Y{ U} \Y{J} } \varepsilon^{2d} \underset{x \in U_\varepsilon }{\sum}\underset{y \in J_\varepsilon}{\sum} 
 \left(\cweighta{\tfrac{x}{\varepsilon}, \tfrac{x-y}{\varepsilon}} -\E \left[\cweighta{\tfrac{x}{\varepsilon}, \tfrac{x-y}{\varepsilon}} \right] \right) \left( \tfrac{\Y{x-y}}{\varepsilon} \right)^{d+ps-\q}  \\
=& I_1^\eps + I_2^\eps,
\end{align*}
where
\begin{align*}
I_1^\eps:=c \varepsilon^{-\alpha} \frac{1}{\Y{ U} \Y{J} }  \varepsilon^{2d} \underset{x \in U_\varepsilon }{\sum}\underset{y \in J_\varepsilon}{\sum} 
 \E \left[\cweighta{\tfrac{x}{\varepsilon}, \tfrac{x-y}{\varepsilon}} \right] \left( \tfrac{\Y{x-y}}{\varepsilon} \right)^{d+ps-\q}
\end{align*}
and
\begin{align*}
I_2^\eps :=& c \varepsilon^{-\alpha} \frac{1}{\Y{ U} \Y{J} } \varepsilon^{2d}  \underset{x \in U_\varepsilon }{\sum}\underset{y \in J_\varepsilon}{\sum} 
 \left(\cweighta{\tfrac{x}{\varepsilon}, \tfrac{x-y}{\varepsilon}} -\E \left[\cweighta{\tfrac{x}{\varepsilon}, \tfrac{x-y}{\varepsilon}} \right] \right) \left( \tfrac{\Y{x-y}}{\varepsilon} \right)^{d+ps-\q}.
\end{align*}
Using $\E \left[\cweighta{\tfrac{x}{\varepsilon}, \tfrac{x-y}{\varepsilon}} \right] =\tilde{C} \varepsilon^{\alpha} \left(\tfrac{\Y{x-y}}{\varepsilon} \right)^{-d-ps+\q}$ 
, $I^\varepsilon_1$ turns into
\begin{align*}
I^\varepsilon_1=& c\tilde{C}  \frac{1}{\Y{ U} \Y{J} } \varepsilon^{2d} \underset{x \in U_\varepsilon }{\sum}\underset{y \in J_\varepsilon}{\sum} 
 \left( \tfrac{\Y{x-y}}{\varepsilon} \right)^{-d-ps+\q} \left( \tfrac{\Y{x-y}}{\varepsilon} \right)^{d+ps-\q} \epsarrow  c\tilde{C} .
\end{align*}

The second step of the proof is to show that $I^\varepsilon_2 \asarrow 0 $.
Let $0<\delta $, then we want to show that
\begin{align}
\lim_{\varepsilon \rightarrow 0}\mathbb{P} \left( \Y{ \frac{c \varepsilon^{2d-\alpha} }{\Y{ U} \Y{J} }  \underset{x \in U_\varepsilon }{\sum}\underset{y \in J_\varepsilon}{\sum} 
 \left(\cweighta{\tfrac{x}{\varepsilon}, \tfrac{x-y}{\varepsilon}} -\E \left[\cweighta{\tfrac{x}{\varepsilon}, \tfrac{x-y}{\varepsilon}} \right] \right) \left( \tfrac{\Y{x-y}}{\varepsilon} \right)^{d+ps-\q}} < \delta \right) =  1.
\end{align}
Therefore, we apply Tschebyscheff's inequality (Theorem \ref{th:tschebyscheff}) with 
\begin{align*}
X= c \varepsilon^{-\alpha} \frac{1}{\Y{ U} \Y{J} } \varepsilon^{2d}  \underset{x \in U_\varepsilon }{\sum}\underset{y \in J_\varepsilon}{\sum} 
\left( \tfrac{\Y{x-y}}{\varepsilon} \right)^{d+ps-\q}
 \cweighta{\tfrac{x}{\varepsilon}, \tfrac{x-y}{\varepsilon}} .
\end{align*}
Since $\var{\sum_{i} a_i Y_i} = \sum_i a_i^2 \var{Y_i}$ for uncorrelated random variables $Y_i$, the variance of $X$ can be estimated as follows 
\begin{align*}
\var{X} &= \frac{c^2 \varepsilon^{-2\alpha + 4 d} }{\Y{ U}^2 \Y{J}^2 }  \underset{x \in U_\varepsilon }{\sum}\underset{y \in J_\varepsilon}{\sum} 
\left( \tfrac{\Y{x-y}}{\varepsilon} \right)^{2(d+ps-\q)}
p_\xi^{\alpha, \q} \left(1 - p_\xi^{\alpha, \q}  \right) \\
&\leq \frac{c^2 \varepsilon^{-2\alpha + 4 d} }{\Y{ U}^2 \Y{J}^2 }  \underset{x \in U_\varepsilon }{\sum}\underset{y \in J_\varepsilon}{\sum} 
\left( \tfrac{\Y{x-y}}{\varepsilon} \right)^{2(d+ps-\q)}
\tilde{C} \varepsilon^\alpha \left( \tfrac{\Y{x-y}}{\varepsilon} \right)^{-(d+ps-\q)} \\
&= \frac{c^2 \tilde{C}}{\Y{ U}^2 \Y{J}^2 }  \varepsilon^{-\alpha + 4 d}  \underset{x \in U_\varepsilon }{\sum}\underset{y \in J_\varepsilon}{\sum} 
\left( \tfrac{\Y{x-y}}{\varepsilon} \right)^{d+ps-\q} \\
&=\frac{c^2 \tilde{C}}{\Y{ U}^2 \Y{J}^2 }  \varepsilon^{-\alpha +2 d} \int_U \int_J
\left( \tfrac{\Y{\Reps x- \Reps y}}{\varepsilon} \right)^{d+ps-\q} ~\dif x \dif y \\
&\leq \frac{c^2 \tilde{C}}{\Y{ U} \Y{J} }  \varepsilon^{d-ps+\q-\alpha} M_{U,J}^{d+ps-\q},
\end{align*}
where $p_\xi^{\alpha, \q}$ is defined in \eqref{eq:defprob} and
\begin{align*}
 M_{U,J} := \max_{x \in U,~y\in J} \Y{x-y}.
\end{align*}
Now we use estimate \eqref{tschebyscheff} of Theorem \ref{th:tschebyscheff} and get
\begin{align}
\nonumber
&\mathbb{P} \left( \Y{c \varepsilon^{-\alpha} \frac{1}{\Y{ U} \Y{J} } \varepsilon^{2d}  \underset{x \in U_\varepsilon }{\sum}\underset{y \in J_\varepsilon}{\sum} 
 \left(\cweighta{\tfrac{x}{\varepsilon}, \tfrac{x-y}{\varepsilon}} -\E \left[\cweighta{\tfrac{x}{\varepsilon}, \tfrac{x-y}{\varepsilon}} \right] \right) \left( \tfrac{\Y{x-y}}{\varepsilon} \right)^{d+ps-\q}} < \delta \right) \\
 &\geq 1-\frac{1}{\delta^2} \frac{c^2 \tilde{C}}{\Y{ U} \Y{J} } M_{U,J}^{d+ps-\q}  \varepsilon^{d-ps+\q-\alpha} ~\xlongrightarrow{\varepsilon \rightarrow 0} 1, \qquad \text{for } d> ps-\q + \alpha.
 \label{eq:conv_prob}
\end{align}

 Consequently, we find $\lim_{\varepsilon \rightarrow 0} \mathbb{P}\left(\Y{I_2^\epsilon }<\delta \right)=1$ for every $\delta>0$, which implies that $I_2^\epsilon \rightarrow 0$ almost surely.
\end{proof}

\begin{proof}[\bf Proof of Theorem \ref{m:l:konvJeps}]
Let $\delta > 0$ and $k\in \mathbb{N}$. We define
\begin{align*}
g_\gamma (x,y) := \begin{cases} \Y{x-y}^{d+ps}, & \text{if~} \Y{x-y} > \gamma \\
\gamma^{d+ps}, & \text{if~} \Y{x-y}\leq\gamma
\end{cases}
\end{align*}
and with the superscript $k$ we denote the component-wise restriction of a function $u$ to the interval $[-k,k]$, defined as
\begin{align*}
    u^k := \left( 
    \begin{array}{c}
\max\{-k, \min\{u_1, k\} \} \\ 
\vdots \\
\max\{-k, \min\{u_d, k\} \}
\end{array}
   \right).
\end{align*}
Then, as we will argue in detail below,
\begin{align*}
\liminf_{\varepsilon \rightarrow 0} \Jepsu &\geq \liminf_{\eps\to0}\int_Q \int_Q \Reps \coeff \frac{\RVk}{g_\gamma (\Reps x, \Reps y)} ~\dif x~\dif y \\
 &= \,\bar{c} \int_Q \int_Q \frac{\Vk}{g_\gamma (x, y)} ~\dif x~\dif y \\
&\xlongrightarrow[k\rightarrow \infty]{\gamma \rightarrow 0} \bar{c} \int_Q \int_Q \frac{\V}{\Y{x-y}^{d+ps}} ~\dif x~\dif y  .
\end{align*}
The convergence in the above equation for $\varepsilon \rightarrow 0$ follows by Lemma \ref{m:l:konvMittelInt} with 
\begin{align*}
    \f_\varepsilon &= \Reps \coeff ,\\
    v_\varepsilon &= \frac{\RVk}{g_\gamma (\Reps x, \Reps y)}, \\
    v &= \frac{\Vk}{g_\gamma (x, y)}.
\end{align*}
By assumption, we have $\Repsu \rightarrow u$ strongly in $L^p$ and thus $\Repsu^k \rightarrow u^k$ strongly in $L^p$. Therefore there exists a subsequence $\varepsilon_j$ such that $\Reps u_{\varepsilon_j}^k \rightarrow u^k$ pointwise a.e. It follows that $v_\varepsilon \rightarrow v$ pointwise a.e.
Condition \eqref{m:eq:konvMittelInt} is ensured by Theorem \ref{m:th:konvMittelw} and we have 
\begin{align*}
    \sup_{Q \times Q} \Y{v_\varepsilon} \overset{\eqref{eq:p_growth}}{\leq}  \frac{c_{xy}2^p k^p}{g_\gamma (\Reps x, \Reps y)} \leq \frac{c_{xy}2^p k^p}{\gamma^{d+ps}} < \infty .
\end{align*}

If additionally $ \sup_\varepsilon \sup_{x,y \in Q_\varepsilon} \frac{\V}{\Y{x-y}^{d+ps}} < \infty $, we can apply Lemma \ref{m:l:konvMittelInt} directly analogous to the previous case.
\end{proof}

\section{Proof of Theorem \ref{th:maintheorem}}

Using Theorem \ref{th:gamma_konv_laplace} and the results from Section \ref{chap:first_model}, proving the main-theorem (Theorem \ref{th:maintheorem}) is reduced to combining the convergences of the individual terms.
\begin{proof}[Proof of Theorem \ref{th:maintheorem}.] 

\begin{enumerate}
\item[(i)] 
If $\liminf_{\eps\to0}\cE_\eps u_\eps=\infty$, the statement is clear. Otherwise let $u_{\eps'}$ be a minimizing subsequence of $u_\eps$, i.e. 
$$\liminf_{\varepsilon\to0} \mathcal{E}_{\varepsilon} u_{\varepsilon}=\lim_{\varepsilon'\to0} \mathcal{E}_{\varepsilon'} u_{\varepsilon'} =: \cE_0 < \infty.$$ 
Without loss of generality, we may assume that $u_\eps=u_{\eps'}$, because any other subsequence $u_{\eps''}$ of $u_\eps$ will have the property
$$\liminf_{\varepsilon''\to0} \mathcal{E}_{\varepsilon''} u_{\varepsilon''}\geq\lim_{\varepsilon'\to0} \mathcal{E}_{\varepsilon'} u_{\varepsilon'} \,,$$
and therefore will not harm the inequalities below.

Theorem \ref{th:gamma_konv_laplace} yields along a subsequence $\varepsilon_j$ that $\mathcal{R}_{\varepsilon_j}^\ast u_{\varepsilon_j}\rightarrow u$ in $L^2(Q)^d$ and $\mathcal{E}^{loc}u \leq \liminf_{\varepsilon_j \rightarrow 0} \mathcal{E}^{loc}_{\varepsilon_j}u_{\varepsilon_j}$. 
Theorem \ref{thm:Poincare} gives us a further (sub)subsequence $\varepsilon_{j_k}$ such that $\mathcal{R}_{\varepsilon_{j_k}}^\ast u_{\varepsilon_{j_k}}\rightarrow u$ in $L^p(Q)^d$ and
applying Theorem \ref{m:l:konvJeps} yields
\begin{align*}
\mathcal{E}u &\leq \liminf_{{\varepsilon_{j_k}} \rightarrow 0} \mathcal{E}^{V}_{\varepsilon_{j_k}}u_{\varepsilon_{j_k}} + 
\liminf_{\varepsilon_{j_k} \rightarrow 0} \mathcal{E}^{loc}_{\varepsilon_{j_k}}u_{\varepsilon_{j_k}}
+  \lim_{\varepsilon_{j_k} \rightarrow 0} \mathcal{F}_{\varepsilon_{j_k}} u_{\varepsilon_{j_k}}  &\text{a.s.} \\
&=  \liminf_{\varepsilon_{j_k} \rightarrow 0} \left(  \mathcal{E}^{V}_{\varepsilon_{j_k}}u_{\varepsilon_{j_k}} + \mathcal{E}^{loc}_{\varepsilon_{j_k}}u_{\varepsilon_{j_k}} +   \mathcal{F}_{\varepsilon_{j_k}} u_{\varepsilon_{j_k}}   \right)  &\text{a.s.}\\
&=  \liminf_{\varepsilon_{j_k} \rightarrow 0}  \mathcal{E}_{\varepsilon_{j_k}} u_{\varepsilon_{j_k}}=\cE_0  &\text{a.s.}
\end{align*}
\item[(ii)] 
First we observe that for $u \in C^2_c(Q)^d$ and $u_\varepsilon := \mathcal{R}_\varepsilon u$ we have $\mathcal{E}_\varepsilon u_\varepsilon \asarrow \mathcal{E}u$.

In general, if $\mathcal{E}u < \infty$  for any function $u\in L^2(Q)^d$, then Korn's inequality implies that $u \in H^1_0(Q)^d$.
In particular, there exists for any $K\in \mathbb{N}$ a function $u_K \in C^2_c(Q)^d$ such that 
\begin{align}\label{eq:prooflimsup1}
    \y{u-u_K}_{H^1(Q)^d}< \frac{1}{2K}
\end{align}
and 
\begin{align}\label{eq:prooflimsup2}
    \Y{\mathcal{E}u - \mathcal{E}u_K}< \frac{1}{2K}.
\end{align}
For a given $K\in \mathbb{N}$ we choose such a $u_K\in C^2_c(Q)^d$ that satisfies properties \eqref{eq:prooflimsup1} and \eqref{eq:prooflimsup2} and we choose an $\varepsilon_K$ such that $\Y{\mathcal{E}_\varepsilon \mathcal{R}_\varepsilon u_K - \mathcal{E} u_K}< \frac{1}{2K}$ for any $\varepsilon < \varepsilon_K <\varepsilon_{K-1}$. 
Therefore, we have 
$$\Y{\mathcal{E} u - \mathcal{E}_\varepsilon \mathcal{R}_\varepsilon u_K} < \frac{1}{K}$$
for any $\varepsilon <\varepsilon_K$.
Setting $u_\varepsilon := \mathcal{R}_\varepsilon u_K$ for $\varepsilon \in [\varepsilon_{K+1},\varepsilon_K)$ the claim follows.
\end{enumerate}
\end{proof}

\section*{Declarations}

\subsubsection*{Funding}
This work was funded by the Deutsche Forschungsgemeinschaft (DFG, German Research Foundation) under Germany's Excellence Strategy - EXC-2193/1 - 390951807 and SPP2256 project HE 8716/1-1, project ID:441154659.

\subsubsection*{Data availability}
We do not generate or analyse data sets as we follow a theoretical mathematical approach.

\subsubsection*{Conflict of interest}
The authors declare there is no conflict of interest.

\subsubsection*{Author contribution}
All authors contributed to the final manuscript.
Patrick Dondl came up with the idea of investigating a model with random long-range interactions that have constant material properties.
Martin Heida and Simone Hermann developed the exact setup of the model as well as the first outline of the proof.
Simone Hermann worked out almost all of the
technical details, performed the proofs and wrote the manuskript. 
Martin Heida added sections \ref{section2.1}, \ref{section2.1.2} and most of section \ref{sec3.1}.
Patrick Dondl and Martin Heida reviewed and edited the manuscript.

\providecommand{\href}[2]{#2}
\providecommand{\arxiv}[1]{\href{http://arxiv.org/abs/#1}{arXiv:#1}}
\providecommand{\url}[1]{\texttt{#1}}
\providecommand{\urlprefix}{URL }

\end{document}